%% file: main.tex
\documentclass[11pt]{amsart}%
\usepackage{amsbsy,amssymb,amscd,amsfonts,latexsym,amstext,
amsmath,epsfig,graphicx,bm}
\usepackage{amsmath,amssymb}
\usepackage{amsthm}
\usepackage[varg]{txfonts}
\usepackage[all]{xy}
\newtheorem{Theorem}{Theorem}[section]
\newtheorem{Corollary}[Theorem]{Corollary}
\newtheorem{Lemma}[Theorem]{Lemma}

\newtheorem{Proposition}[Theorem]{Proposition}

\newtheorem{TheoremAlph}{Theorem}

\newtheorem{CorNoNum}{Corollary}

\newtheorem{Definition}[Theorem]{Definition}

\newtheorem{Conjecture}[Theorem]{Conjecture}

\newcommand{\Zz}{\mathbb{Z}/{2}}
\newcommand{\p}{\vec{p}}
\newcommand{\Zps}{\mathbb{Z}}

\newcommand{\R}{\mathbb{R}}
\newcommand{\X}{\mathcal{X}}

\newcommand{\J}{J}

\newcommand{\Hz}[2][*]{H^{#1}(#2;\Zz)}

\newcommand{\Ker}{\mathbf{Ker}}

\newcommand{\Image}{\mathbf{Im}}
\newcommand{\SO}{\textit{SO}}
\newcommand{\Gr}[1]{G_{#1}}
\newcommand{\oriGr}[1]{\widetilde{G}_{#1}}
\newcommand{\oriw}{\bar{w}}

\renewcommand{\vec}[1]{\mathbf{#1}}

\newcommand{\LT}{\mathrm{LT}}
\newcommand{\LCM}{\mathrm{LCM}}

\newcommand{\setS}{\mathcal{S}}
\newcommand{\setP}{\mathcal{P}}

\newcommand{\Poly}[1]{P_{#1}}
\newcommand{\lscat}{\mathrm{cat}}
\newcommand{\cuplen}{\mathrm{cup}}
\newcommand{\zcuplen}{\mathrm{cup}_{\Zz}}
\newcommand{\bracket}[1]{\langle #1\rangle}

\begin{document}
\title[Gro\"{o}bner bases of oriented Grassmann manifolds]
{Application of Gr\"{o}bner bases to the cup-length of 
oriented Grassmann manifolds}
\author{Tomohiro Fukaya}
\thanks{The author is supported by Grant-in-Aid for JSPS Fellows (19 3177) from Japan Society for the Promotion of Science.}
\address{Department of Mathematics Kyoto University. Kyoto 606-8502, Japan.}
\email{tomo\_xi@math.kyoto-u.ac.jp}
\subjclass[2000]{Primary~ 55M30, Secondary~57T15,13P10}
\keywords{Cup-length; LS-category; Gr\"{o}bner bases; Immersion}
\input{intro}
\input{coh}
\input{gen}
\input{grob}

\input{imm}

\bibliographystyle{amsalpha}
\bibliography{refs}

\end{document}

%% file: intro.tex
\begin{abstract}
For $n=2^{m+1}-4\, (m\geq 2)$, we determine the cup-length of 
$\Hz{\oriGr{n,3}}$ by finding a Gr\"{o}bner basis associated with a
 certain subring, 
 where $\oriGr{n,3}$ is the oriented Grassmann manifold $\SO(n+3)/\SO(n)
 \times \SO(3)$. As its applications, we provide not only a lower but
 also an upper bound for the LS-category of $\oriGr{n,3}$. We also study
 the immersion problem of $\oriGr{n,3}$.



\end{abstract}
\maketitle

\section{Introduction}
\label{sec:intro}
Let $R$ be a commutative ring. The cup-length of $R$ is defined by the
greatest number $n$ such that there exist $x_1,\dots,x_n \in R \setminus
R^{\times}$ with $x_1\cdots x_n \neq 0$.
We denote the cup-length of $R$ by $\cuplen(R)$. In particular, for a
space $X$ and a commutative ring $A$, the cup-length of $X$ with the
coefficient $A$, is defined by 
$\cuplen(\tilde{H}^*(X;A))$. We denote it by $\cuplen_A(X)$.
It is well-known that $\cuplen_A(X)$ is a lower bound for the LS-category of
$X$.

The aim of this paper is to study $\zcuplen(\oriGr{n,3})$, where
$\oriGr{n,k}$ is the oriented Grassmann manifold $\SO(n+k)/\SO(n) \times
\SO(k)$. Note that $\oriGr{n,k}$ is $(nk)$-dimensional.
While the cohomology of $\oriGr{n,2}$ is well-known, that of
$\oriGr{n,3}$ is in vague. 
However, Korba\v{s} \cite{K} gave rough estimations
for $\zcuplen(\oriGr{n,3})$ by considering the height of 
$w_2 \in \Hz{\oriGr{n,3}}$, where $w_2$ is the second Stiefel-Whitney class.

The author studies $\Hz{\oriGr{n,3}}$ by considering Gr\"{o}bner bases
associated with a certain subring of $\Hz{\oriGr{n,3}}$.
It seems that, in principle, the method of Gr\"{o}bner bases works
better in such complicated calculations than that of usual algebraic
topology. 
The author employs a computer and carries a huge amount of
calculations for finding the above Gr\"{o}bner bases and then he dares
to conjecture: 
\begin{Conjecture}
\begin{eqnarray*}
   \cuplen_{\Zz}(\oriGr{n,3}) = 
 \begin{cases}
  2^{m+1} -3 & \text{ when } 
    2^{m+1} -4 \leq n \leq 2^{m+1} + 2^{m} - 6,\\
  2^{m+1} -3 + k&  \text{ when } 
    n = 2^{m+1} + 2^{m} - 5 + k, \; 
    0 \leq k \leq 2,\\
{2^{m+1} + 2^m +\dots }   & \text{ when } 
    {n =  2^{m+1} + 2^m + \dots + 2^{j-1} -2  + k}, \\
\hspace{10mm}{+ 2^{j+1} +2^{j-1} + k} 
   & \hspace{40mm}{0 \leq k \leq 2^j -1}.
 \end{cases}
\end{eqnarray*}
\end{Conjecture}
When $n=2^{m+1}-4\,(m\geq 2)$, our method works very well and we obtain:
\begin{TheoremAlph}
\label{thA:cuplen}
$\cuplen_{\Zz}(\oriGr{n,3}) = n+1$ when 
$n=2^{m+1} -4 \, (m\geq 2)$.
\end{TheoremAlph}
By a dimensional reason, we have 
\begin{eqnarray}
\label{eq:cat}
\lscat(X) \leq \frac{3}{2}n,
\end{eqnarray}
where $\lscat(X)$ denotes the LS-category of a space $X$ normalized as
$\lscat(*) = 0$. Theorem A gives not only lower bounds for 
$\lscat(\oriGr{n,3})$, but
also refines the inequality (\ref{eq:cat}). Actually we obtain:
\begin{CorNoNum}
\label{corNo:lscat}
$n+1 \leq \lscat(\oriGr{n,3}) < \frac{3}{2}n$ when 
$n= 2^{m+1}-4 \,(m\geq 2)$.
In particular, we have  $\lscat(\oriGr{4,3}) = 5$.
\end{CorNoNum} 
We will give applications of Theorem \ref{thA:cuplen} for the immersion
problem of 
$\oriGr{n,3}$. By the classical result  of Whitney \cite{Wh}, we
know that $\oriGr{n,3}$ immerses into $\R^{6n-1}$. We will show:
\begin{TheoremAlph}
\label{thA:imm}
 The oriented Grassmann manifold $\oriGr{n,3}$ immerses into $\R^{6n-3}$
 but not into $\R^{3n+8}$ when $n= 2^{m+1}-4 \,(m\geq 3)$ and
 $\oriGr{4,3}$ immerse into $\R^{21}$ but not into $\R^{17}$.
\end{TheoremAlph}
Remark : Walgenbach \cite{Wa} obtained better results on the
non-immersion of $\oriGr{n,3}$:
$\oriGr{n,3}$ does not immerses into $\R^{4n-2m +3}$. On the other hand,
due to R. Cohen \cite{C}, $\oriGr{n,3}$ is known to be immersed into
$\R^{6n-m+1}$. Then Theorem B gives a better estimation when $m=2,3$.

The organization of this paper is as follows. In section
\ref{sec:coh}, we consider the double covering map 
$p_n \colon \oriGr{n,3} \rightarrow \Gr{n,3}$, 
where $\Gr{n,3}$ is the unoriented Grassmann manifold 
$O(n+3)/O(n) \times O(3)$.
We identify the subring $\Image p_n^*$ of $\Hz{\oriGr{n,3}}$ with
a certain algebra $\Zz[\oriw_2,\oriw_3]/J_{n}$, where generators of
$J_{n}$ are given.
In section \ref{sec:gen}, setting $n=2^{m+1}-4 \,(m\geq 2)$, we will give an
explicit description of generators of the ideal $J_{n}$ by using the
binary expansion.
In section \ref{sec:grob}, we compute a Gr\"{o}bner basis of $J_{n}$
and obtain $\cuplen(\Image p_n^*)$.
In section \ref{sec:cuplen}, we show $\cuplen(\Image p_n^*)$
determines $\cuplen_{\Zz}(\oriGr{n,3})$ and obtain it. As its
applications, we give some estimations for $\lscat(\oriGr{n,3})$ and
study the immersion problem of $\oriGr{n,3}$.

%% file: coh.tex
\section{Cohomology of $\oriGr{n,3}$}
\label{sec:coh}
We consider the double covering 
\begin{eqnarray}
\label{eq:cover}
 p_{n} \colon \oriGr{n,3}
\rightarrow \Gr{n,3}.
\end{eqnarray} 
It will be shown that $\cuplen_{\Zz}({\oriGr{n,3}})$ can be
determined by $\cuplen(\Image p_{n}^*)$. Then we shall investigate 
$\cuplen(\Image p_{n}^*)$.
%

The mod 2 cohomology of $BO(3)$ is given by
\[
 \Hz{BO(3)} = \Zz[w_1,w_2,w_3],
\]
where $w_i$ is the $i$-th universal Stiefel-Whitney class. It is
well-known that the canonical map $i \colon \Gr{n,3} \rightarrow BO(3)$
induces an epimorphism $i^* \colon \Hz{BO(3)} \rightarrow
\Hz{\Gr{n,3}}$. Hereafter we denote $i^*(w_i)$ by the same symbol
$w_i$ ambiguously.

One can easily see that the above double covering (\ref{eq:cover})
induces the Wang sequence as:
\[
\cdots
\longrightarrow \Hz[q-1]{\Gr{n,3}} \stackrel{\cdot w_1}{\longrightarrow} 
 \Hz[q]{\Gr{n,3}} \stackrel{p_{n}^*}{\longrightarrow} \Hz[q]{\oriGr{n,3}} 
\longrightarrow \cdots.
\]
Then we have 
\[
 \Image p_{n}^* \cong \Zz[w_1,w_2,w_3]\Big/(w_1,\Ker i^*).
\]

Let $\pi \colon \Zz[w_1,w_2,w_3] \rightarrow \Zz[w_2,w_3]$ be the
abstract ring homomorphism defined by $\pi(w_1) = 0$, $\pi(w_2) = w_2$
and $\pi(w_3) = w_3$. Then it induces the isomorphism
\[
 \Image p_{n}^* \cong \Zz[\oriw_2,\oriw_3] \Big/ J_{n},
\]
where $\pi(\Ker i^*) = J_{n}$ and we denote $w_i$ in 
$\Hz{\oriGr{n,3}}$ by $\oriw_i$. 
Note that the commutative diagram
\begin{eqnarray*}
 \xymatrix{
\oriGr{n,3} \ar[d]_{\tilde{\imath}} \ar[r]_{p_n} & \Gr{n,3} \ar[d]_i\\
B\SO(3)  \ar[r]_{p_\infty} & BO(3)}
\end{eqnarray*}
yields that $\tilde{\imath}^*(w_i) = \oriw_i$ for $i= 2,3$ and $p_\infty^*
\colon \Hz{BO(3)} \rightarrow \Hz{B\SO(3)}$ is expressed by 
$\pi \colon \Zz[w_1,w_2,w_3] \rightarrow \Zz[w_2,w_3]$.

Let us give explicit generators of $J_{n}$.
Borel \cite{B} showed that $\Ker i^*$ is generated by
the homogeneous components of degrees $n+1$, $n+2$ and $n+3$ in
\[
 \frac{1}{1+w_1+ w_2 + w_3}.
\]
Then it follows that $J_{n}$ is generated by the homogeneous components of
degrees $n+1$, $n+2$ and $n+3$ in
\[
 \frac{1}{1+ \oriw_2 + \oriw_3}.
\]
Let $N$ be the unique integer which satisfies $2^N < n \leq 2^{N+1}$. 
Since $\dim{\oriGr{n,3}} < 4n \leq 2^{N+3}$, we have
\[
 (1+\oriw_2+\oriw_3)^{2^{N+3}}=1 
\]
in $\Hz{\oriGr{n,3}}$. Then it follows that 
\begin{eqnarray*}
 \frac{1}{1+\oriw_2+\oriw_3} = (1+\oriw_2+\oriw_3)^{2^{N+3}-1}
\end{eqnarray*}
and hence $J_{n}$ is generated by 
\begin{eqnarray}
\label{eq:g}
  g_r=\sum_{ \frac{r}{3} \leq s \leq \frac{r}{2}}
      \binom{s}{3s-r}\oriw_2^{3s-r}\oriw_3^{r-2s}
\end{eqnarray}
for $r = n+1,n+2,n+3$.

%% file: gen.tex
\section{Investigating generators of $J_{n}$}
\label{sec:gen}
In this section, we investigate generators $g_{n+1}$, $g_{n+2}$ and 
$g_{n+3}$ of $J_{n}$ by exploiting the binary expansion.

Let us prepare notation for the binary expansion. To a non-negative
integer $x$ with $0 \leq x < 2^k$, we assign a sequence
\[
\epsilon_k(x) = (x_{k-1},\dots,x_0) \in \{0,1\}^k 
\] 
such that
\begin{eqnarray}
\label{eq:binary}
 x = \sum_{i=0}^{k-1}x_i2^i.
\end{eqnarray}
(\ref{eq:binary}) is, of course, the binary expansion of $x$.
We denote $1 - a$ by $\overline{a}$ with $a \in \{0,1\}$. 
For example, we have 
\[
 \epsilon_k(2^{k}-1)=(1,\dots,1)
\]
and
\[
 \epsilon_{k}(2^k-1-x)=(\overline{x_{k-1}},\dots,\overline{x_{0}})
\]
for $\epsilon(x) = (x_{k-1},\dots,x_0)$.
We often denote $(x_k,\dots,x_0)\in \{0,1\}^k$ by $\vec{x}_k$.

To calculate $\binom{s}{3s-r}$ modulo 2, we use the following well-known
result from elementary number theory.

\begin{Lemma}
\label{lem:lucas}
Let $n$ and $k$ be non-negative integers such that $k\leq n\leq 2^l-1$ and
 $\epsilon_{l}(n)=(n_{l-1},\ldots,n_0),\epsilon_l(k)=(k_{l-1},\ldots,k_0)$. 
Then we have
 $\binom{n}{k}\equiv 1 \pmod 2$ if and only if $k_i=1$ 
 implies $n_i=1$ for each $i$.
\end{Lemma}
In the rest of this paper, we assume that 
\par
\begin{center}
\fbox{$n = 2^{m+1} -4 \; (m\geq2)$.  } 
\end{center}
Applying Lemma \ref{lem:lucas} to the coefficients of $g_{n+1}$, we have:
\begin{Proposition}
\label{prop:g_{n+1}}
 $\binom{s}{3s-(n+1)}$ is even for all integer  $s$ with 
$ \frac{n+1}{3} \leq s \leq \frac{n+1}{2}$, that is $g_{n+1} = 0$.
\end{Proposition}
\begin{proof}
 Let $\epsilon_m(s) = (s_{m-1},\dots,s_0)$ for $\frac{n+1}{3}\leq s \leq
 \frac{n+1}{2}$ and let $\epsilon_{m+1}(n+1-2s)=(t_m,\dots,t_0)$.
Since $s \leq 2^m -2$, there exists an integer $i$ such that $s_i=0$.
Let $i$ be the least integer satisfying $s_i = 0$, that is,
$\epsilon_m(s) = (s_{m-1},\dots,s_{i+1},0,1,\dots,1)$.
Then it is easy to show that $t_i=1$. 
Hence it follows from Lemma \ref{lem:lucas} that 
$\binom{s}{3s-(n+1)} = \binom{s}{n+1-2s} \equiv 0 \pmod 2$.
%
\end{proof}
Next we investigate  $g_{n+2}$. Coefficients of $g_{n+2}$ are well
understood by considering their binary expansion as in the above case of
$g_{n+1}$. Let
\label{def:Sm,Pm}
\begin{eqnarray*}
\setS_k = \left\{s \in \Zps \left|
\tfrac{n(k)+2}{3} \leq s \leq \tfrac{n(k)+2}{2} \text{, }
\epsilon_k(s) = (s_{k-1},\dots,s_0) \,
\text{\small satisfies that if }s_j = 0,\text{\small then }s_{j+1} = 1
 \right. \right\},
\end{eqnarray*}
here $n(k)=2^{k+1}-4$.
Note that $\frac{n(k)+2}{3} \leq s \leq \frac{n(k)+2}{2}$ implies that
$s_{k-1}$ is always equal to 1 for each $s \in \setS_k$ with $\epsilon_k(s) =
(s_{k-1},\dots, s_0)$. 
There is a one-to-one correspondence between non-zero coefficients of 
$g_{n+2}$ and $\setS_m$ as:
\begin{Lemma}
\label{lem:g_{n+2}}
 \[
  \binom{s}{3s - (n+2)} \equiv 1 \pmod 2 \quad \text{ if and only if}
  \quad s \in \setS_m.
 \]
\end{Lemma}
\begin{proof}
 Let $\epsilon_m(s) = (s_{m-1},\dots,s_0)$ for 
$\frac{n(m)+2}{3} \leq s \leq \frac{n(m)+2}{2}$.
Then we have
\begin{eqnarray*}
 \epsilon_{m+1}(n+2 - 2s) =
      (\overline{s_{m-1}},\dots,\overline{s_{0}},0)
\end{eqnarray*}
and hence Lemma \ref{lem:g_{n+2}} follows from Lemma \ref{lem:lucas}.
\end{proof}
It is convenient for calculations in section \ref{sec:grob} to index
coefficients of $g_{n+2}$ by exponents of $\oriw_2$ in (\ref{eq:g}), that
is, $3s - (n+2)$, not by $s \in \setS_m$. Then we define a set 
$\setP_k$ by
\begin{eqnarray*}
 \setP_k = \left\{ \left. p \in \Zps \right|
	       p = 3s - (n(k)+2), \, s \in \setS_k
\right\}.
\end{eqnarray*}
$\setP_m$ is expressed by the binary expansion as:

\begin{Proposition}
\label{prop:Pm}
Let
\begin{eqnarray*}
 \Delta_k = \left\{(p_{k-1},\dots,p_0) \in \{0,1\}^k \left|\text{If }
p_{l-1}=1 \text{ and }p_l=p_{l+1}=\cdots=p_{l+2t}=0, 
\text{ then } p_{l+2t+1}=0 \right.\right\},
\end{eqnarray*}
here we assume that $p_{-1} = 1$. Then we have
\[
 \setP_m = \left\{\left. p \in \Zps \right| \epsilon_m(p) \in \Delta_m\right\}.
\]
\end{Proposition}
We list some properties of $\Delta_{k}$ which will be
useful in the following discussion. The proof is straightforward.
\begin{Proposition}
\label{prop:Delta}
The set $\Delta_k$ has the following properties.
\renewcommand{\theenumi}{\alph{enumi}}
\begin{enumerate}
 \item \label{enu:(1,p_k)}
       $\vec{p}_k \in \Delta_k$ implies $(1,\vec{p}_k) \in
       \Delta_{k+1}$.
 \item \label{enu:(p_k,1)} 
       $\vec{p}_k \in \Delta_k$ implies $(\vec{p}_k,1) \in \Delta_{k+1}$.
 \item \label{enu:D_m}
       $\Delta_m 
       = \left\{\left.(1,\vec{p}_{m-1})\in \{0,1\}^m \right|
         \vec{p}_{m-1} \in \Delta_{m-1}\right\}
    \sqcup
         \left\{\left.(0,0,\vec{p}_{m-2})\in \{0,1\}^m \right|
         \vec{p}_{m-2} \in \Delta_{m-2}\right\}$.
\end{enumerate}
\end{Proposition}

\begin{proof}[Proof of Proposition \ref{prop:Pm}]
Let $s \in \setS_m$ with $\epsilon_{m}(s)=(s_{m-1},\dots,s_0)$.
If $s_{m-2} = 0$, then one has $s_{m-1}=1$ and $s_{m-3}=1$ by definition 
of $\setS_m$. Then one can easily see that 
$(s_{m-2},\dots,s_{0}) \in \setS_{m-3}$. If $s_{m-2}=1$, then one can
 see that $(s_{m-2},\dots,s_{0}) \in \setS_{m-1}$ as well. Hence one has
 obtained 
\begin{eqnarray}
\label{eq:S}
 \setS_{m}=\left\{\left. s+2^{m-1} \right| s \in \setS_{m-1}\right\}
  \sqcup \left\{\left. s+2^{m-1} \right| s \in \setS_{m-2}\right\}. 
\end{eqnarray}
We will show Proposition \ref{prop:Pm} by induction.
We suppose that it is true for $m-1$ and $m-2$.
Let $s \in \setS_{m-1}$ and $p=3(s+2^{m-1}) - (n+2)$.
By the hypothesis of the induction,  
$\epsilon_{m-1}(3s -(n'+2)) = \vec{p}_{m-1} \in \Delta_{m-1}$,
where $n' = 2^{m}-4$.
Since 
\begin{eqnarray*}
  p = 3(s+2^{m-1}) - (n+2) 
     =  3s - 2^{m} +2 +2^{m-1} = 3s - (n'+2) + 2^{m-1},
\end{eqnarray*}
we have
\[
 \epsilon_m(p)=(1,\vec{p}_{m-1}) \in \Delta_{m}.
\]
Similarly, let $s \in \setS_{m-2}$ and $p=3(s+2^{m-1})-(n+2)$.
By the hypothesis of the induction,  
$\epsilon_{m-2}(3s -(n''+2)) = \vec{p}_{m-2} \in \Delta_{m-2}$, 
where $n''=2^{m-1}-4$.
Since 
\begin{eqnarray*}
  p = 3(s+2^{m-1}) - (n+2) 
     = 3s - 2^{m-1} +2 = 3s - (n''+2),
\end{eqnarray*}
we have
\[
 \epsilon_m(p)=(0,0,\vec{p}_{m-2}) \in \Delta_{m}.
\]
Thus, by (\ref{eq:S}), we obtain
\[
\setP_m = 
   \left\{\left. p+2^{m-1} \right| p \in \setP_{m-1}\right\}
 \sqcup \setP_{m-2}
\]
and, by (\ref{enu:D_m}) of Proposition \ref{prop:Delta},
we have established Proposition \ref{prop:Pm}.

\end{proof}
For the last of this section, we investigate $g_{n+3}$.
Coefficients of $g_{n+3}$ can be well understood by using the binary
expansion as well as above. Let  
\label{def:S'm,Qm}
\begin{eqnarray*}
\setS'_k = \left\{s' \in \Zps \left|
\tfrac{n(k)+3}{3} \leq s' \leq \tfrac{n(k)+3}{2} \text{, }
\epsilon_k(s') = (s_{k-1},\dots,s_1,1)\,
\text{\small satisfies that if }s_j = 0,\text{\small then }s_{j+1} = 1
 \right. \right\}.
\end{eqnarray*}
Quite similarly to Lemma \ref{lem:g_{n+2}}, we can see:
\begin{Lemma}
\label{lem:g_{n+3}}
 \begin{eqnarray*}
  \binom{s'}{3s' - (n+3)}
     \equiv 1 \; \pmod{2} \text{ if and only if } s' \in \setS'_m.
 \end{eqnarray*}
\end{Lemma}
We give an explicit description of the set
\begin{eqnarray*}
\setP'_k = \left\{ \left. p' \in \Zps \right|
	       p' = 3s' - (n(k)+3), \, s' \in \setS'_k
\right\}
\end{eqnarray*}
as well. Define a map
\[
 \iota:\setS_{m-1} \rightarrow \setS'_m
\]
by $\iota(s) = 2s+1$. Then, obviously, it is bijective. Note that, for
$s' = \iota(s)$, 
\[
 3s' - (n+3) = 3\iota(s) -2^{m+1} + 1 = 6s - 2^{m+1} + 4 = 2(3s -
 (n'+2))
\]
where $n' = 2^{m} -4$.
Then we have $p \in \setP_{m-1}$ if and only if $p' \in \setP'_m$ such that 
$\epsilon_m(p') = (\vec{p}_{m-1},0)$ for 
$\epsilon_{m-1}(p) = \vec{p}_{m-1}$. Hence we have obtained:
\begin{Proposition}
\label{prop:P'm}
$ \setP'_m = \left\{\left. p \in \Zps  \right|
 \epsilon_m(p)=(\vec{p}_{m-1},0), \; 
 \vec{p}_{m-1}\in \Delta_{m-1}\right\}.$
\end{Proposition}

%% file: grob.tex
\section{Gr\"{o}bner basis and cup-length}
\label{sec:grob}
In this section, by using the result of the previous section, we search
for a Gr\"{o}bner basis of $J_{n}$ in order to determine
$\cuplen(\Image p_n^*)$.
\subsection{Gr\"{o}bner bases}
We first recall the definition and some facts of Gr\"obner bases by
restricting to our specific case.
In order to clarify our discussion and to simplify notation, we shall make a
convention of identifying a two variable polynomial ring with a certain
set as follows. Let $\X=\{(p,q)\in\Zps^2|p\geq 0,q\geq 0\}$ and
let $P[\X]$ denote the set of finite subset of $\X$. 
By assigning $F\in P[\X]$ to $\sum_{(p,q)\in
F} \oriw_2^p \oriw_3^q$, we can identify $P[\X]$ with a polynomial
ring $\Zz[\oriw_2, \oriw_3]$ and we shall make this identification
throughout this section. This identification translates the operations
in $\Zz[ \oriw_2, \oriw_3]$ into $P[\X]$ as: For $F,G\in P[\X]$,
\[
 F+G=F\cup G\setminus F\cap G,
\]
\[
 F\cdot G=\sum_{(p,q)\in F,\,(r,s)\in G}(p+r,q+s).
\]
This translation of operations enables us to handle the following polynomial
calculations easily.

The order of
$\X$ is given by the usual lexicographic order. Namely, for
$(p,q),(r,s)\in\X$,
\[
 (p,q)\ge(r,s)\text{ if and only if }p>r\text{ or }p=r,q\ge s.
\]
By employing this order, we search for a Gr\"obner basis of
the ideal $J_n\subset P[\X]$. 

In order to define Gr\"obner bases, we prepare some notation and
terminology. The leading term of a polynomial $F\in
P[\X]$ is the monomial
\[
 \LT(F)=\max\{(p,q)\in F\}.
\]
If there is a monomial $(p,q)\in\X$ such that $(p,q)\cdot\LT(G)\in F$,
then the polynomial $F-(p,q)\cdot\LT(G)$ is called the remainder of $F$
on division by $G$. 
We denote the remainder $R=F-(p,q)\cdot\LT(G)$ of $F$ on division by
$G$, by
\[
 F \xrightarrow{G_*} R.
\]

Choose $F_1,\ldots,F_s\in P[\X]$ and give them an arbitrary order. Then
it is known that there is an
algorithm to provide the decomposition of $F\in P[\X]$ as
\[
 F=A_1F_1 + \cdots+ A_sF_s +R
\]
such that $A_1,\ldots,A_s\in P[\X]$ and $R$ is a linear
combination of monomials, none of which is 
divisible by each $\LT(F_1),\ldots,\LT(F_s)$. The above $R$ is called
the remainder of $F$ on division by $\{F_1,\ldots,F_s\}$ as well. 
However, this decomposition depends on the choice of an order of
$F_1,\ldots,F_s$
and $F\in (F_1,\ldots,F_s) $ does not imply the 
remainder $R=0$. We can overcome this difficulty of remainders 
by choosing a Gr\"obner basis defined as:
\begin{Definition}
Let $I$ be an ideal of $P[\X]$. A finite subset
 $G=\{G_1,\ldots,G_s\}$ is a Gr\"obner basis of $I$ if
\[
 (\{\LT(F)|F\in I\})
  = ( \LT(G_1),\ldots,\LT(G_s)).
\]
\end{Definition}
\begin{Theorem}
\label{Groebner}
Let $I$ be an ideal of $P[\X]$ and let $\{G_1,\ldots,G_s\}$ be 
 a Gr\"obner basis of $I$. Then the remainder of $F \in I$ on division
 by $\{G_1,\ldots,G_s\}$ is zero.
\end{Theorem}

Buchberger \cite{CLO} gave a criterion for a set of polynomials being a
Gr\"{o}bner basis of the ideal generated by it as follows.
For $F,G\in P[\X]$, the least common multiple of $F$ and $G$ is the
monomial
\[
 \LCM(F,G)=(\max\{p,r\},\max\{q,s\}),
\]
where $\LT(F)=(p,q)$ and $\LT(G)=(r,s)$. The $S$-polynomial of $F$ and
$G\in P[\X]$ is 
\[
S(F,G)=\frac{\LCM(F,G)}{\LT(F)}F+\frac{\LCM(F,G)}{\LT(G)}G.
\]

\begin{Theorem}[\cite{CLO}]
\label{Buchberger}
The set of polynomials $\{G_1,\ldots,G_s\}\subset P[\X]$ is a Gr\"obner
 basis of the ideal $(G_1,\ldots,G_s)$ if and only if the remainder of
 $S(G_i,G_j)$ on division by $\{G_1,\ldots,G_s\}$ is zero for each $i\ne j$.
\end{Theorem}
\subsection{Search for a Gr\"obner basis of $\J_{n}$}
The author found the following polynomials experimentally by a computer
calculation. For non-negative integers $i,t$ with 
$t - 2(2^m - 2^i) \equiv 0 \pmod 3$, we
define a polynomial $P(t,i)$ by
\begin{alignat*}{1}
 P(t,i) &= 
\Biggl\{\left(p,\tfrac{t-2p}{3} \right) \in \X \Bigg|
\epsilon_m(p) = (\p_{m-i},\overbrace{0,\dots,0}^{i}), \;
  \p_{m-i} \in \Delta_{m-i} \Biggr\}, \\
 \Poly{i} &= P(2^i + n + 1,i).
\end{alignat*}
%
%
We shall prove that $\{P_0,\ldots,P_m\}$ is a Gr\"obner basis of $J_{n}$.

In order to investigate $P_i$, we define the following sets which will
be useful for expression. Let $\Delta(i,j,l)$ and $\bar\Delta(i,l)$ be
\label{def:Delta2}
 \begin{alignat*}{1}
  \Delta(i,j,l) &=
   \Big\{(\p_{m-j},\p_{j-i-l},\overbrace{1,\dots,1}^{l},
   \overbrace{0,\dots,0}^{i}) \in \{0,1\}^m
   \Big|(\p_{m-j},\p_{j-i-l}) \in \Delta_{m-i-l},\, \p_{j-i-l}
   \neq (1,\dots,1) \Big\},\\
  \bar{\Delta}(i,l) &=
   \Big\{(\p_{m-i-l-2},0,0,\overbrace{1,\dots,1}^{l},
   \overbrace{0,\dots,0}^{i}) \in \{0,1\}^m
   \Big|\p_{m-i-l-2} \in \Delta_{m-i-l} \Big\}.
 \end{alignat*}
It is easy to check:
\begin{Lemma}
\label{lem:Delta}
\[
 \Delta(i,j,l) = \bar{\Delta}(i,l) \sqcup \Delta(i,j,l+1).
\]
\end{Lemma}
Let us begin investigating $P_i$. It is easy to verify that
\begin{eqnarray}
\label{eq:LT(P)}
\LT(P_i)=(2^m-2^i,2^i-1).
\end{eqnarray}
\begin{Proposition}
\label{prop:I=(P)}
 We have $\Poly{0},\dots,\Poly{m} \in \J_{n}$.
In particular $\Poly{0} = g_{n+2}$, $\Poly{1}= g_{n+3}$.
\end{Proposition}
\begin{proof}
By Proposition \ref{prop:Pm} and Proposition \ref{prop:P'm},
 one has $P_0=g_{n+2}$ and $P_1=g_{n+3}$. 
For $i<j$, it follows from (\ref{eq:LT(P)}) that

\begin{alignat*}{1}
S(\Poly{i},\Poly{j}) 
 =\,&  (0,2^j - 2^i) \cdot \Poly{i}+ (2^j-2^i,0) \cdot \Poly{j}\\
 =\,& \Biggl\{\left.\left(p,q_{i,j}\right)
     \in \X \right|\epsilon_m(p) = (\p_{m-j},\p_{j-i},
      \overbrace{0,\dots,0}^{i}), \;
     (\p_{m-j},\p_{j-i}) \in \Delta_{m-i}\Biggr\}\\
 &+ \Biggl\{\left.\left(p,q_{i,j}\right)
    \in \X \right|\epsilon_m(p)= (\p_{m-j},
 \overbrace{1,\dots,1}^{j-i},\overbrace{0,\dots,0}^{i}), \;
     \p_{m-j} \in \Delta_{m-j}\Biggr\}\\
 =\,& \left\{\left.\left(p,q_{i,j}(p)\right)
 \in \X \right|\epsilon_m(p) \in \Delta(i,j,0)\right\},
\end{alignat*}
where
\[
  q_{i,j}(p)= \frac{3\cdot 2^j -2\cdot 2^i + n + 1-2p}{3}.
\]
By the definition of
 $\Delta_k$, one can easily see that
 $\Delta(i,0,i+1)=\Delta_{m-i-2}$. Then it follows that
$S(P_i,P_{i+1})=P_{i+2}$
and hence we have established Proposition \ref{prop:I=(P)}.
\end{proof}
%
We calculate the remainders of $S(\Poly{i},\Poly{j})$ on
division by $\{P_0,\ldots,P_m\}$.
\begin{Lemma}
 \label{lem:Q->Q}
The remainder of $Q_{i,j,l}=\left\{\left.\left(p,q_{i,j}(p)\right)
     \in \X \right|\epsilon_m(p) \in \Delta(i,j,l)\right\}$ 
on division by $\Poly{i+l+2}$ is $Q_{i,j,l+1}$.
\end{Lemma}
\begin{proof}
Let $p(i,l)$ be
 $\epsilon_m(p(i,l))=(\overbrace{1,\ldots,1}^{m-i-l-2},0,0,
 \overbrace{1,\ldots,1}^l,\overbrace{0,\ldots,0}^i)$. 
Then it is easy to see
\begin{eqnarray*}
\LT(Q(i,j,l))=
\left(p(i,l),q_{i,j}(p(i,l))\right)
\end{eqnarray*} 
and it follows from (\ref{eq:LT(P)}) that
\begin{eqnarray*}
 \LT(\Poly{i+l+2}) =
 \left(p(i+l,0),q_{i+j+2,i+j+2}(p(i+l,0))\right).
\end{eqnarray*}
Hence we have
\begin{eqnarray*}
(2^{i+l}-2^i,2^j-2^{i+l+1}) \cdot \LT(\Poly{i+l+2}) = Q(i,j,l).
\end{eqnarray*}
On the other hand, one can easily check that
\begin{eqnarray*}
(2^{i+l}-2^i,2^j-2^{i+l+1}) \cdot \Poly{i+l+2}
 = \left\{\left.\left(p,q_{i,j}(p)\right)
     \in \X \right|\epsilon_m(p)   \in \bar{\Delta}(i,l)\right\}
\end{eqnarray*}
and then it follows from Lemma \ref{lem:Delta} that
\begin{eqnarray*}
Q(i,j,l)
\xrightarrow{{\Poly{i+l+2}}_*}
Q(i,j,l)
+(2^{i+l}-2^i,2^j-2^{i+l+1}) \cdot \Poly{i+l+2}
=
Q(i,j,l+1).
\end{eqnarray*}
\end{proof}
\begin{Theorem}
\label{Groebner-I}
 The set $\{\Poly{0},\dots,\Poly{m}\}$ is a Gr\"obner basis of 
$\J_{n}$.
\end{Theorem}
\begin{proof}
 By Proposition \ref{prop:I=(P)}, we have $J_{n}=(
 \Poly{0},\dots,\Poly{m} )$.
As in the proof of Proposition \ref{prop:I=(P)}, we have 
$S(\Poly{i},\Poly{j})=Q(i,j,0)$ and
 then it follows from Lemma \ref{lem:Q->Q} that, for $i<j$,
\[
 S(\Poly{i},\Poly{j})=Q(i,j,0) \xrightarrow{{\Poly{i+2}}_*}
 Q(i,j,1) \xrightarrow{{\Poly{i+3}}_*} \cdots
 \xrightarrow{{\Poly{j}}_*}
 Q(i,j,j-i-1) \xrightarrow{{\Poly{j+1}}_*} 0.
\]
\end{proof}
\subsection{Cup-length of $\Image p_n^*$}
In order to determine $\cuplen(\Image p_n^*)$, 
let us introduce new polynomials.
For non-negative integers $i,j,s$ with $s-2^{m+1}+2^{i+1} \equiv 0 \pmod 3$,
we define a polynomial $\hat{P}(s,i,j)$ by
\begin{eqnarray*}
   \hat{P}(s,i,j) = 
  \left\{\left.\left(p,\tfrac{s-2p}{3}\right)
     \in \X \right|\epsilon(p) 
     \in \bar{\Delta}(i,j) \right\} .
 \end{eqnarray*}
Then we have
\begin{eqnarray}
\label{eq:P^}
\LT(\Poly{i}) 
         = P(2^i+n+1,i) + P(2^{i}+n+1,i+2) 
	 + \sum_{1 \leq j \leq m-i-2}\hat{P}(2^i+n+1,i,j).  
\end{eqnarray}

In order to investigate $\cuplen(\Image p_n^*)$, we shall calculate
$\min\{p|(p,0)\cdot \LT(\Poly{i}) \in \J_{n}\}$ for each $i$ 
as follows.
\begin{Lemma}
\label{lem:alpha-P}
Let $\alpha_i = \min\{\alpha|(\alpha,0)\cdot P(t,i) \in \J_{n}\}$
 for non-negative integers $i,t$ with 
\begin{eqnarray*}
 2^{i-2} + n + 1 \leq t < 2^{i} + n + 1, \, 
      t-2(2^m-2^i) \equiv 0 \pmod 3.
\end{eqnarray*}
Then we have
$  \alpha_i = 2^m-2^{i-1}.$ In particular, $\alpha_{i}$ is independent 
from $t$ as above.
\end{Lemma}
\begin{proof}
Note that
 \begin{alignat*}{1}
  (2^{i-1},0)\cdot P(t,i) =\,& 
  \Biggl\{\left.\left(p,\tfrac{t+2^{i}-2p}{3}\right)
     \in \X \right|
  \epsilon_m(p) = (\p_{m-i}1,0,\dots,0) \in \{0,1\}^m, \,  
     (\p_{m-i},1) \in \Delta_{m-i+1}\Biggl\},\\
  \left(0,\tfrac{t+2^{i-1} -n-1}{3}\right)\cdot \Poly{i-1} =\,&
  \Biggl\{\left.\left(p,\tfrac{t+2^{i}-2p}{3}\right)
     \in \X \right|
  \epsilon_m(p) = (\p_{m-i+1},0,\dots,0) \in \{0,1\}^m, \, 
     \p_{m-i+1} \in \Delta_{m-i+1}\Biggl\}.
 \end{alignat*}
By Proposition \ref{prop:Delta}, we have
\[
(2^{i-1},0)\cdot P(t,i) 
 + \left(0,\tfrac{t+2^{i-1} -n-1}{3}\right)\cdot \Poly{i-1} 
 = P(t+2^{i},i+1)
\]
and hence
\[
 (2^{i-1},0) \cdot P(t,i) \xrightarrow{{\Poly{i-1}}_*}
 P(t+2^{i},i+1). 
\]
Then we obtain
\begin{eqnarray*}
 (2^{i-1},0) \cdot P(t,i) &\xrightarrow{{\Poly{i-1}}_*}&
   P(t+2^{i},i+1),\\
  (2^{i},0)\cdot P(t+2^{i},i+1) &\xrightarrow{{\Poly{i}}_*}&
   P(t+2^i+2^{i+1},i+2),\\
  &\vdots&\\
  (2^{m-1},0)\cdot P(t+2^{i}+\cdots +2^{m-1},m) 
   &\xrightarrow{{\Poly{m-1}}_*}& 0
 \end{eqnarray*}
and this completes the proof of Lemma \ref{lem:alpha-P}.
\end{proof}

\begin{Lemma}
\label{lem:hat}
Let $\alpha_i$ be as in Lemma \ref{lem:alpha-P}. Then we have
$ (\alpha_i,0)\cdot \hat{P}(s,i,j) \in \J_{n}$.
\end{Lemma}
\begin{proof}
Quite similarly to the proof of Lemma \ref{lem:alpha-P}, one has
\[
   (2^{j+i}+2^i,0) \cdot \hat{P}(s,i,j) 
 + \left(0,\tfrac{s+2^{i+1} + 1}{3}\right) \cdot \Poly{j+i+1}
 = P(s+2^{j+i+1} + 2^{i+1}, j+i+3).
\]
By Lemma \ref{lem:alpha-P}, we have
$2^{j+i}+2^i + \alpha_{j+i+3} < \alpha_i$ 
and then Lemma \ref{lem:hat} is accomplished.
\end{proof}
It follows from Lemma \ref{lem:alpha-P} and Lemma 
\ref{lem:hat} that:
\begin{Proposition}
\label{prop:height}
Let $\alpha_i$ be as in Lemma \ref{lem:alpha-P}. Then we have
$\alpha_{i+1} = 
\min\{\alpha|(\alpha,0)\cdot \LT(\Poly{i}) \in \J_{n}\}$.
\end{Proposition}
\begin{Corollary}
\label{cor:chi_2}
Let $\chi_1$ be a fixed integer such that $2^{m+1}-2^{i+1}-2^{i+2} \leq
 \chi_1 < 2^{m+1}-2^{i}-2^{i+1}$ and let $\chi_2 = \max\{z|(\chi_1,z)\notin
 \J_{n}\}$.
Then we have $\chi_2= 2^i -1$.
\end{Corollary}
\begin{proof}
 Let $(p_i,q_i)=\LT(\Poly{i})$. Then, by (\ref{eq:LT(P)}) and Lemma
 \ref{lem:alpha-P}, we have
\begin{eqnarray*}
&\cdots > p_{i-1}+\alpha_{i+1}> p_i+\alpha_{i+2} > p_{i+1}+\alpha_{i+3}
\cdots,&\\
&\cdots < q_{i-1}<q_i<q_{i+1}< \cdots.&
\end{eqnarray*}
Hence, by Theorem \ref{Groebner} and Theorem \ref{Groebner-I}, we have
 established that, for
$p_{i+1}+\alpha_{i+3}=2^{m+1}-2^{i+1}-2^{i+2}
\leq \chi_1 <2^{m+1}-2^i-2^{i+1} = p_i+\alpha_{i+2}$,
one has $\chi_2=q_{i+1}-1=2^{i+1}-2$.
\end{proof}
From Corollary \ref{cor:chi_2}, $\chi_1 + \chi_2$ takes the maximum when 
$(\chi_1,\chi_2)=(n,0)$ and it is $n$, of course, it is equal to 
$\cuplen(\Image p_n^*)$. 
Therefore we have obtained:
\begin{Corollary}
 \label{cor:cup(Im p^*)}
$\cuplen(\Image p_n^*) = n$. In particular, $\oriw_2^n \neq 0$.
\end{Corollary}

%% file: imm.tex
\section{Cup-length of $\oriGr{n,3}$ and its applications}
\label{sec:cuplen}
In this section, we determine $\cuplen_{\Zz}(\oriGr{n,3})$ and
give its applications to immersion of $\oriGr{n,3}$ into a Euclidean
space.

\begin{proof}[Proof of Theorem \ref{thA:cuplen}]
 By Corollary \ref{cor:cup(Im p^*)}, one has $\bar{w}_2^n\ne 0$. Then,
 by Poincar\'e 
 duality, there exists $x\in \Hz[n]{\oriGr{n,3}}$ such that
 $\bar{w}_2^nx\ne 0$ and hence we have
 $\cuplen_{\Zz}(\oriGr{n,3})\ge n+1$.

 Note that the canonical map $\oriGr{n,3}\to B\SO(3)$ is
 an $n$-equivalence. Then it follows that $\Hz{\oriGr{n,3}}\cong\mathbf{Im}
 p_n^*$ in dimensions less than $n$. Now suppose that there exist
 $x_1,\ldots,x_{n+2} \in \widetilde{H}^*(\oriGr{n,3};\Zz)$ such that $x_1\cdots
 x_{n+2}\ne 0$. By a dimensional reason, one has $|x_i|<n$ for each $i$
 and then this contradicts to Corollary \ref{cor:cup(Im p^*)}. Hence we
 have obtained  Theorem \ref{thA:cuplen}.
\end{proof}
%
%
\begin{proof}[Proof of Corollary] From Theorem \ref{thA:cuplen} and the 
inequality
 $\cuplen_{\Zz}(\oriGr{n,3}) \le \lscat(\oriGr{n,3})$, it
 follows that $n+1 \leq \lscat(\oriGr{n,3})$.

Note that $\bar{w}_2\in \Hz[2]{\oriGr{n,3}}$ is the fundamental class in
 the sense of James \cite{J}. By Corollary \ref{cor:cup(Im p^*)}, we
 have $\bar{w}_2^{n+1}=0$
 and then it follows from Proposition 5.3 in \cite{J} that
 $\lscat(\oriGr{n,3})<\frac{3}{2}n$. Hence we have established
 Corollary.
\end{proof}

Let us consider the immersion of $\oriGr{n,3}$ into a Euclidean space as
applications of Theorem \ref{thA:cuplen}. Of course, as mentioned in
section \ref{sec:intro}, we know, by the
result of Whitney \cite{Wh}, that $\oriGr{n,3}$ immerses into $\R^{6n-1}$. We
shall give a slightly better estimation. 

We denote the canonical vector bundle over
$\oriGr{n,3}$ by $\gamma$ and a stable normal
bundle of $\oriGr{n,3}$ by $\nu$. We abbreviate the classifying map
$\oriGr{n,3}\to B\SO(\infty)$ of $\nu$ by the same symbol $\nu$. It is
well-known that $T\oriGr{n,3} = \gamma \otimes \gamma^\bot$, then we have 
\begin{alignat}{1}
 T\oriGr{n,3} \oplus \gamma \otimes \gamma 
	    &= \gamma \otimes \gamma^\bot \oplus \gamma \otimes \gamma \notag\\
	    &= \gamma \otimes (\gamma^\bot \oplus \gamma ) \notag\\
	    &= (n+3)\gamma.
\end{alignat}
By Corollary \ref{cor:cup(Im p^*)}, we have 
$(1+\oriw_2+\oriw_3)^{n+4} = 1$. 
Using the formula for the Stiefel-Whitney class of a tensor product
shows that $w(\gamma \otimes \gamma) = 1+\oriw_2^2+\oriw_3^3$.
Since $\nu \oplus T\oriGr{n,3}$ is trivial, we have
\begin{alignat}{1}
\label{normal}
w(\nu)  &= \frac{w(\gamma \otimes \gamma)}{w((n+3)\gamma)} \notag \\
        &= \frac{1+\oriw_2^2+\oriw_3^2}{(1+\oriw_2 +\oriw_3)^{n+3}} \notag \\
        &= (1+\oriw_2^2+\oriw_3^2)(1+\oriw_2+\oriw_3) \notag\\
        &= 1 + \oriw_2 +\oriw_3 + \oriw_2^2+\oriw_2^3 +\oriw_3^2 
           +  \oriw_2^2\oriw_3+ \oriw_2\oriw_3^2 + \oriw_3^3. 
\end{alignat}
Then it immediately follows that $\oriGr{n,3}$ does not immerse into
$\R^{3n+8}$ for $n=2^{m+1}-4 \,m\geq3$ and $\oriGr{4,3}$ does not immerse into
$\R^{17}$.

Now let us consider the modified Postnikov tower of the fibration
\[
 B\SO(3n-3)\to B\SO(\infty)
\]
following Gitler and Mahowald \cite{GM}. The
$A_2$-free resolution of $H^*(SO(\infty)/SO(3n-3))$ in dimensions less
than or equal to $3n$ is given as follows, where $A_2$ denotes the mod
$2$ Steenrod algebra.
\[
 C_2\stackrel{d_2}{\to} C_1\stackrel{d_1} {\to}
 H^*(\SO(\infty)/\SO(3n-3)) \to 0,
\]
\begin{eqnarray*}
&C_1=\bracket{x_{3n-3},x_{3n-1}},\;C_2=\bracket{y_{3n-1}},&\\
& \;d_1(x_{3n-3})=e_{3n-3},
 \;d_1(x_{3n-1}) = e_{3n-1},
 \;d_2(y_{3n-1})=Sq^2x_{3n-3},\;|x_i|=i,&
\end{eqnarray*}
where $\bracket{x}$ and $e_{i}$ denote the free $A_2$-module
generated by
$x$ and a generator of 
\[
 H^{i}(\SO (\infty)/ \SO(3n-3))\cong\Zz
\] for $i = 3n-3,3n-1$ respectively. Then the modified Postnikov tower
of $B\SO(3n-3)\to B\SO(\infty)$ in dimensions less than or equal to $3n$
is given as: 
\begin{eqnarray*}
 \xymatrix{
       B\SO(3n-3)\ar[d]\\
       E\ar[d]\ar[r]^{k_2\phantom{-----}} & K(\Zz,3n-1)\\
     B\SO(\infty)\ar[r]^{w_{3n-2}\times w_{3n} \phantom{---------}} 
    & K(\Zz,3n-2) \times K(\Zz,3n)
 }
\end{eqnarray*}
It follows from \eqref{normal} that $w_{3n-2}(\nu) = w_{3n}(\nu)=0$ and then
$\nu \colon \oriGr{n,3} \to B\SO (\infty)$ lifts to 
$\tilde{\nu} \colon \oriGr{n,3} \to E$. 
By Poincar\'{e} duality, one has $\Hz[3n-1]{\oriGr{n,3}}=0$. Then
$\tilde{\nu} \colon \oriGr{n,3}\to E$ lifts to $\bar{\nu} \colon \oriGr{n,3}\to
B\SO(3n-3)$ and hence we can see from the result of Hirsch \cite{H} that
$\oriGr{n,3}$ immerses into $\R^{6n-3}$. Then we have one obtains
Theorem \ref{thA:imm}.